\documentclass[12pt]{amsart}

\DeclareMathOperator{\cof}{\cof}

\usepackage[english]{babel}
\usepackage[utf8]{inputenc}
\usepackage{amsmath}
\usepackage{amssymb}
\usepackage{amsfonts}
\usepackage{amsthm}
\usepackage{mathrsfs}
\usepackage[all]{xy}
\usepackage[pdftex]{graphicx}
\usepackage{color}
\usepackage[pdftex,dvipsnames]{xcolor} 
\usepackage{cite}
\usepackage{url}
\usepackage{indent first}
\usepackage[labelfont=bf,labelsep=period,justification=raggedright]{caption}
\usepackage[english]{babel}
\usepackage[utf8]{inputenc}
\usepackage[colorlinks=true, allcolors=blue]{hyperref}
\usepackage[colorinlistoftodos]{todonotes}
\usepackage{tkz-fct}
\usepackage{mathdots}
\usepackage{comment}
\usepackage{float}

\usepackage{algpseudocode}
\usepackage{mathrsfs}
\usepackage{bbm}
\usepackage{dsfont}
\usepackage{theoremref}
\usepackage[dvipsnames]{xcolor}
\usepackage{enumitem}
\usepackage{caption}
\usepackage{subcaption}

\usetikzlibrary{automata,positioning}

\theoremstyle{plain}

\theoremstyle{plain}
\newtheorem{thm}{Theorem}

\newtheorem{lemma}[thm]{Lemma}
\newtheorem{cor}[thm]{Corollary}
\newtheorem{conj}[thm]{Conjecture}
\newtheorem{prop}[thm]{Proposition}

\newtheorem{qn}[thm]{Question}

\theoremstyle{definition}
\newtheorem{defn}[thm]{Definition}

\theoremstyle{remark}

\numberwithin{equation}{section}
\numberwithin{thm}{section}

\usepackage{mathrsfs}
\usepackage{calc} 
\usetikzlibrary{lindenmayersystems}
\graphicspath{ {./images/} }

\newtheorem*{remark}{Remark}
\numberwithin{equation}{section}
\numberwithin{thm}{section}

\newcommand{\R}{\mathbb{R}}
\newcommand{\C}{\mathbb{C}}
\newcommand{\N}{\mathbb{N}}

\newcommand{\be}{\begin{equation}}
\newcommand{\ee}{\end{equation}}

\pgfdeclarelindenmayersystem{Cantor dust}{
    \symbol{S}{\pgfpathrectangle{\pgfpointorigin}{\pgfpoint{\pgflsystemcurrentstep}{\pgflsystemcurrentstep}}}
    \symbol{U}{\pgftransformyshift{\pgflsystemcurrentstep}}
    \symbol{R}{\pgftransformxshift{\pgflsystemcurrentstep}}
    \rule{S -> [UUSRRS][SRRS]}
    \rule{U -> UUU}
    \rule{R -> RRR}
}

\def \Al{0.4 pt}
\def \Size{40 pt}

\pgfdeclarelindenmayersystem{Cantor dust Alpha}{
    \symbol{S}{\pgfpathrectangle{\pgfpointorigin}{\pgfpoint{{\Al * \Size}}{{\Al * \Size}}}}
    \symbol{U}{\pgftransformyshift{{\Al * \Size}}}
    \symbol{P}{\pgftransformyshift{{\Size-(\Size*2*\Al)}}}
    \symbol{R}{\pgftransformxshift{{\Al * \Size}}}
    \symbol{K}{\pgftransformxshift{{\Size-( \Size*2*\Al)}}}
    \rule{S -> [UPSRKS][SRKS]}
    \rule{U -> UPU}
    \rule{R -> RKR}
    \rule{P -> PPP}
    \rule{K -> KKK}
}

\begin{document}

\title[Conformally Removable Subsets]{Every Planar Set has a Conformally Removable Subset with the Same Hausdorff Dimension}
\date{\today}

\author[Drillick]{Hindy Drillick}
\address[Hindy Drillick]{
     Columbia University \\ New York, NY, 10027, USA.}
\email{hdrillick@math.columbia.edu}


\maketitle

\begin{abstract}
In this paper we show that given any compact set $E \subset \hat{\C}$, we can always find a conformally removable subset with the same Hausdorff dimension as $E$. 

\end{abstract}

\section{Introduction}

We begin with defining conformal removability. Recall that a function between planar domains is conformal if it is holomorphic and invertible. Let $\hat{\C}$ denote the Riemann Sphere.

\begin{defn}[Conformal Removability] A compact set $E \subset \hat{\C}$ is conformally removable if for all homeomorphisms $f: \hat{\C} \to \hat{\C}$, if $f$ is conformal on $\hat{\C} \backslash E$, then $f$ is conformal on $\hat{\C}$.
\end{defn}

Note that $E$ is removable if and only if all homeomorphisms of $\hat{\C}$ that are conformal off of $E$ are actually M\"{o}bius transformations since these are precisely the conformal homeomorphisms of the Riemann sphere. See \cite{MR3429163} for a discussion of notions of removability for other classes of mappings, including removability for conformal maps that are not necessarily homeomorphisms of $\hat{\C}$.

The simplest example of a conformally removable set is a finite collection of points. Finite collections of points have zero measure, and in fact all conformally removable sets must have zero measure (see the remark after Proposition 4.4 in \cite{MR3429163}). However, it is possible to construct conformally removable sets that are large in the sense of Hausdorff dimension. Indeed, it is well known that conformally removable sets of any Hausdorff dimension in the interval $[0,2]$ exist. We will give an explicit construction of such sets and use them to prove the following theorem:

\begin{thm}[Main Theorem] \thlabel{main1}
Given an arbitrary compact set $E \subset \hat{\C}$, there exists a compact set $E' \subset E$ that is conformally removable such that $\dim_H E' = \dim_H E$.
\end{thm}

\textbf{Acknowledgements:} This work was done as part of my senior thesis at Stony Brook University. I would like to thank Professor Christopher Bishop and Dimitrios Ntalampekos for suggesting this problem and supervising my thesis. I would like to thank Malik Younsi for his helpful comments.

\section{John Domains} \label{Section:John}
In this section we construct conformally removable Cantor sets of any desired Hausdorff dimension in the open interval $(0,2)$. We start by recalling the definition of a John domain and a theorem by Peter Jones giving a sufficient condition for conformal removability.

\begin{defn}[John domain]
A domain $\Omega \subset \hat{\C}$ is a John domain with center $z_0$ if there exists $\varepsilon > 0$ such that for all $z_1 \in \Omega$ there exists a path $\gamma: [0,1] \to \Omega$ with $\gamma(0) = z_0$ and  $\gamma(1) = z_1$ satisfying the following: for all $t \in [0,1] $,
\be \label{john-general}
d(\gamma(t), \partial \Omega) \geq \varepsilon d(\gamma(t), z_1),
\ee
where $d$ is the chordal distance.
\end{defn}

\begin{thm}[Jones \cite{MR1315551}, Corollary 1] \thlabel{John-Domain-Removable}
The boundary of a John domain $\Omega \subset \hat{\C}$ is conformally removable.
\end{thm}

We show that certain self similar Cantor sets in the plane are boundaries of John domains and hence are conformally removable. Given $0< \alpha < \frac{1}{2}$, let $C_\alpha$ be the following Cantor set. Let $A_0 = I \times I$ be the unit square. Given a collection of squares $A_n$ of side-length $\alpha^n$, form $A_{n+1}$ by replacing each square by four congruent subsquares of side length $\alpha^{n+1}$ as shown in Figure \ref{fig:Cantor}. Define \be C_\alpha = \bigcap_{n=0}^{\infty} A_n. \ee

\begin{figure}
 \hspace*{-2cm}     
  \begin{tikzpicture}[scale = 2]
    \fill[lindenmayer system={Cantor dust,axiom=S,step=36pt,order=0}, opacity = .3]lindenmayer system;
    \begin{scope}[shift={(2,0)}]
       \fill[lindenmayer system={Cantor dust,axiom=S,step=12pt,order=1}, opacity = .3]lindenmayer system;
    \end{scope}
    \begin{scope}[shift={(0,-2)}]
       \fill[lindenmayer system={Cantor dust,axiom=S,step = 4pt,order=2}, opacity = .3]lindenmayer system;
    \end{scope}
    \begin{scope}[shift={(2,-2)}]
       \fill[lindenmayer system={Cantor dust,axiom=S,step=4/3pt,order=3}, opacity = .3]lindenmayer system;
    \end{scope}
\end{tikzpicture}
    \caption{The first four iterations of the construction of $C_\alpha$}
    \label{fig:Cantor}
\end{figure}
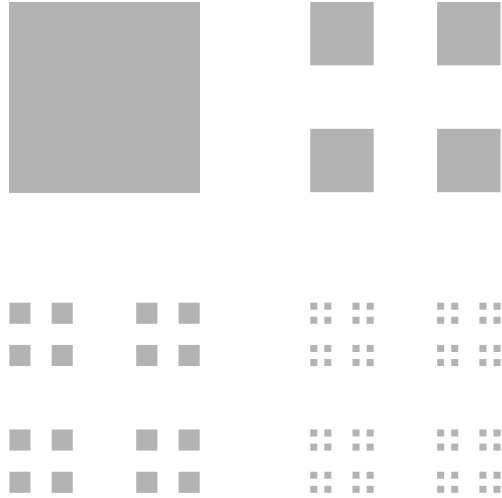

\begin{prop} \thlabel{positive-Hausdorff-measure}
The Hausdorff dimension of $C_\alpha$ is $d = -\frac{\log4}{\log\alpha}.$ Furthermore, \be \mathscr{H}^{d}(C_\alpha) > 0, \ee where $\mathscr{H}^d$ is the $d$-dimensional Hausdorff measure. 
\end{prop} 

For a proof see Example 1.2.9 and Lemma 1.2.8 in \cite{MR3616046}. Since we can choose any $\alpha \in (0, \frac{1}{2})$, this construction gives us Cantor sets with all Hausdorff dimensions in the open interval $(0,2)$, but gives no set with dimension equal to zero or two.

\begin{prop} \thlabel{Cantor-is-John}
$\Omega = \hat{\C} \backslash C_\alpha$ is a John domain.
\end{prop}

\begin{proof}
This follows from Theorem 4.2 in \cite{MR1953266}. We also give a brief sketch of a direct proof here. We can let $z_0 = \infty$. All points $z_1$ outside the unit square trivially satisfy \eqref{john-general}. For points $z_1$ inside the unit square it suffices to construct paths satisfying \eqref{john-general} from base-points on the boundary of a larger square concentric to the unit square, which can then easily be connected to $z_0$. Note that due to the equivalence of the chordal and Euclidean metrics in any bounded region, we can now replace the chordal distances in condition \eqref{john-general} with Euclidean distances. 

Let $Q \subset A_n$ be an $n$th generational square in the construction of $C_\alpha$ with width $\alpha^n$. Let $\gamma_Q$ be the boundary of the concentric square of width $\alpha^n +\frac{\alpha^{n-1}(1-2\alpha)}{2}$ (this extra bit is half of the distance between $Q$ and its neighboring squares). Let $R_Q$ denote the closed region between $\gamma_Q$ and the four curves $\gamma_{P_i}$ contained inside $Q$ that correspond to the four children squares $P_i$ of $Q$. $R_Q$ has the property that any point inside it has distance comparable to $\alpha^n$ from $C_\alpha$ (more explicitly, this distance is greater than $\frac{\alpha^{n}(1-2\alpha)}{2}$) and can be connected to $\gamma_Q$ by a curve in $R_Q$ of length comparable to $\alpha^n$. Using induction and the geometric formula, we can connect any point in $R_Q$ to a point on $\gamma_{I \times I}$ by a path satisfying the John condition. Every point in $(I \times I) \setminus C_\alpha$ belongs to some $R_Q$, so we are done.
\end{proof}

\begin{cor} \thlabel{Cantor-removable}
$C_\alpha$ is conformally removable.
\end{cor}
\begin{proof}
This follows immediately from \thref{John-Domain-Removable}.
\end{proof}

The sets $C_\alpha$ give us conformally removable sets of all Hausdorff dimensions in $(0, 2).$ A point is an example of a conformally removable set with Hausdorff dimension zero. An example of a conformally removable set of Hausdorff dimension two will be given by \thref{main1} in the case where $\dim_H E = 2.$

\section{Dimensions of Intersections}
We are interested in proving that every planar compact set $E$ contains a removable subset $E'$ of the same Hausdorff dimension as itself. We have already constructed conformally removable Cantor sets of arbitrary Hausdorff dimensions in $(0,2)$. To construct $E'$ we will position these sets in such a way that their intersection with $E$ still has large Hausdorff dimension.

When two manifolds $M^k$ and $N^l$ intersect transversally in an $n$-dimensional ambient space, we know that their intersection is a submanifold with dimension $k+l-n$. The following theorem by Mattila gives an analogous formula for the Hausdorff dimension of the intersection of two fractals. 

Let $A, B \subset \R^n$ be Borel sets. We want to know what can be said about the intersection of $A$ with the images of $B$ under different isometries. More formally, let $G$ be the group of isometries. Every transformation $\sigma \in G$ can be decomposed into a translation followed by an orthogonal transformation. Let $\tau_z$ denote the translation by $z \in \R^n$, and let $g \in O(n)$ be the orthogonal transformation, where $O(n)$ is the $n$-dimensional orthogonal group. The space of translations is just $\R^n$ and can be given the standard Lebesgue measure $\mathcal{L}^n$. $O(n)$ is a compact group and we will denote its Haar probability measure by $\Theta_n.$ Recall that $\mathscr{H}^\alpha$ denotes the $\alpha-$dimensional Hausdorff measure. 

\begin{thm}[Mattila \cite{MR3617376}, Theorem 7.4] \thlabel{Mattila}
Suppose $0 < s < n, \  0 < t < n, \  s+t > n$ and $t > (n+1)/2$. If $A$ and $B$ are Borel sets in $\R^n$ with $\mathscr{H}^s(A) > 0$ and $\mathscr{H}^t(B) > 0$, then for $\Theta_n$ almost all $g \in O(n)$,
\be \label{dimension-inequality}
\mathcal{L}^n(\{z \in \R^n: \dim_H (A \cap (\tau_z \circ g)(B)) \geq s + t - n\}) > 0.
\ee
\end{thm}

\begin{remark}
An earlier version of this paper used Theorem 8.3 of \cite{MR3236784} (8.2 in earlier editions), which is similar to \thref{Mattila} except that it omits the hypothesis of positive Hausdorff measure on $A$ and $B$. However, it turns out that this result is incorrect; a counterexample is given in \cite{BDN}, which also corrects an application of the result to Falconer's distance set problem for polygonal norms.
\end{remark}

\section{Proof of Theorem}

The following proposition demonstrates that the property of being conformally removable is a local one. Given a compact set $E$, we will use this proposition to construct the desired subset $E'$.

\begin{prop}[Younsi \cite{MR3552551}, Proposition 11]\thlabel{locality}
The following are equivalent: 
\begin{enumerate}
\item For any open set $U$ with $E \subset U$, every homeomorphism  \\ $f:U \to f(U)$ that is conformal on $U \backslash E$ is actually conformal on the whole open set $U$. 
\item $E$ is conformally removable.
\end{enumerate}
\end{prop}

We are now ready to begin proving our main theorem.

\begin{lemma}\thlabel{annuli}
Consider a compact set $E \subset \C$ with positive Hausdorff dimension and any positive sequence $d_n$ that is strictly increasing to $d=\dim_H E$. Then there exists a sequence of nested and concentric open squares $S_n$ such that the annuli $R_n=S_n \setminus S_{n+1}$ satisfy \be
\mathscr{H}^{d_n}(R_n \cap E) > 0.
\ee
\end{lemma}

\begin{proof}
There exists some point $p \in E$ such that for all $\varepsilon,$ \be \dim_H(B(p, \varepsilon) \cap E) = d. \ee If not, then every point in $E$ would be contained in an open subset of $E$ with dimension strictly less than $d$. Since $E$ is compact, we can take a finite subcover of these sets, and it would follow that the dimension of $E$ is strictly less than $d$, which is a contradiction. 

Let $S_1$ be any open square centered at $p$. Suppose we have already been given an open square $S_n$ centered at $p.$ By the choice of $p$, \be \dim_H(S_n \cap E) = d > d_n. \ee By Frostman's lemma (See \cite{MR3617376}, Theorem 2.7 or \cite{MR3616046}, Lemma 3.1.1 for the statement and proof of this lemma), there is a Frostman measure $\mu$ with support in $S_n \cap E$ such that $\mu(B(x,r)) \leq r^{d_n}.$ Since a Frostman measure cannot assign positive measure to the single point $p$, there must exist some open square $S_{n+1} \subset S_n$ centered at $p$ such that $\mu(S_n \backslash S_{n+1}) > 0.$ Then the annulus $R_n= S_n \backslash S_{n+1} $ satisfies $\mathscr{H}^{d_n}(R_n \cap E) > 0$. 
\end{proof}

For the convenience of the reader, we restate the main theorem, \thref{main1}
\begin{thm}[Main Theorem] \thlabel{main}
Given an arbitrary compact set $E \subset \hat{\C}$, there exists a compact set $E' \subset E$ that is conformally removable such that $\dim_H E' = \dim_H E$.
\end{thm}

\begin{proof}
We can assume $E \subset \C$ since if $E$ is the whole sphere then we can just replace $E$ by any compact subset lying in $\C$ with Hausdorff dimension $2$. Let $d = \dim_H E.$ If $d=0$ then we can choose $E'$ to be any point in $E$, and we are done since a single point is conformally removable and has Hausdorff dimension zero. Otherwise, choose a positive sequence $d_n$ that is strictly increasing to $d$, and construct a sequence of annuli $R_n$ as in \thref{annuli} with center $p \in E$. Denote the width of $R_n$ by $q_n.$ Our goal is to place a conformally removable Cantor set in every other annulus so that the copies do not overlap, and so that they intersect $E$ with increasing dimension. Let $b_n$ be a sequence such that $2 - d_n < b_n < 2,$ $b_n > \frac{3}{2},$ and such that $b_n \to 2$. For each $n \in \N$, let $G_{2n}$ be a conformally removable Cantor set with Hausdorff dimension $b_{2n}$ as constructed in Section \ref{Section:John}, re-scaled so that its diameter is smaller than the minimum of $\frac{q_{2n+1}}{2}$ and $\frac{q_{2n-1}}{2}$. We have by construction that \be \mathscr{H}^{d_{2n}}(R_{2n} \cap E) > 0. \ee We also have that \be \mathscr{H}^{b_{2n}}(G_{2n}) > 0 \ee by \thref{positive-Hausdorff-measure}. Finally, we have that $d_{2n} + b_{2n} > 2,$ and that $b_{2n} > \frac{3}{2}$. Therefore, by \thref{Mattila}, there exists some $z \in \R^2$ and some $g \in O(2)$ such  that 
\be \label{dimension-inequality-annuli}
\dim_H(R_{2n} \cap E \cap (\tau_z \circ g)(G_{2n})) \geq d_{2n} +  b_{2n} - 2.
\ee
Denote $ (\tau_z \circ g)(G_{2n})$ by $G'_{2n}$ (See Figure \ref{fig:Annuli}). The set $G'_{2n}$ must overlap with the annulus $R_{2n},$ and its diameter is smaller than half the width of either of the two neighboring odd-numbered annuli. Therefore, it will not overlap with $G'_{2m}$ for any $m \neq n$.

\begin{figure}
\centering
\def\svgscale{.5}
\begingroup%
  \makeatletter%
  \providecommand\color[2][]{%
    \errmessage{(Inkscape) Color is used for the text in Inkscape, but the package 'color.sty' is not loaded}%
    \renewcommand\color[2][]{}%
  }%
  \providecommand\transparent[1]{%
    \errmessage{(Inkscape) Transparency is used (non-zero) for the text in Inkscape, but the package 'transparent.sty' is not loaded}%
    \renewcommand\transparent[1]{}%
  }%
  \providecommand\rotatebox[2]{#2}%
  \newcommand*\fsize{\dimexpr\f@size pt\relax}%
  \newcommand*\lineheight[1]{\fontsize{\fsize}{#1\fsize}\selectfont}%
  \ifx\svgwidth\undefined%
    \setlength{\unitlength}{544.15484557bp}%
    \ifx\svgscale\undefined%
      \relax%
    \else%
      \setlength{\unitlength}{\unitlength * \real{\svgscale}}%
    \fi%
  \else%
    \setlength{\unitlength}{\svgwidth}%
  \fi%
  \global\let\svgwidth\undefined%
  \global\let\svgscale\undefined%
  \makeatother%
  \begin{picture}(1,0.94256316)%
    \lineheight{1}%
    \setlength\tabcolsep{0pt}%
    \put(0,0){\includegraphics[width=\unitlength,page=1]{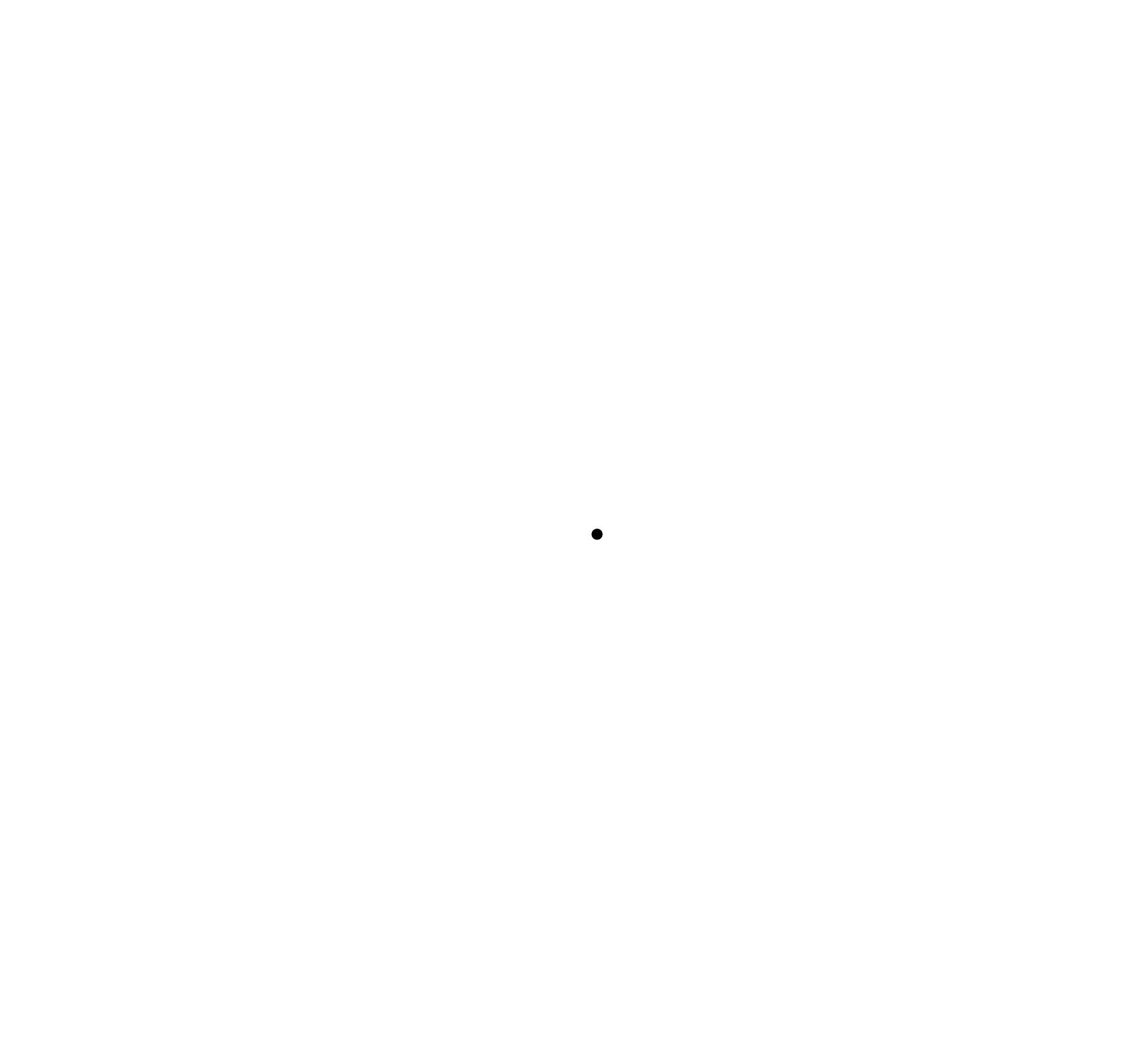}}%
    \put(0.53564465,0.45711261){\color[rgb]{0,0,0}\makebox(0,0)[lt]{\lineheight{1.25}\smash{\begin{tabular}[t]{l}p\end{tabular}}}}%
    \put(0,0){\includegraphics[width=\unitlength,page=2]{annuli3.pdf}}%
    \put(-0.00288578,0.43338254){\color[rgb]{0,0,0}\makebox(0,0)[lt]{\lineheight{1.25}\smash{\begin{tabular}[t]{l}E\end{tabular}}}}%
    \put(0,0){\includegraphics[width=\unitlength,page=3]{annuli3.pdf}}%
  \end{picture}%
\endgroup%

\caption{Annuli centered at $p \in E$ with the sets $G'_{2n}$ placed in every other annulus and overlapping $E$} 
\label{fig:Annuli}
\end{figure}

Let 
\be 
G= \overline{\bigcup_{n =1}^{\infty}G'_{2n}}  = \bigcup_{n =1}^{\infty}G'_{2n} \cup \{p\}.
\ee
$G$ is clearly compact. Suppose that we are given a homeomorphism $f: \hat{\C} \to \hat{\C}$ that is conformal on $\hat{\C} \backslash G$. By construction, we can find disjoint neighborhoods $U_n$ around each $G_{2n}$. Then $f$ is conformal on $U_n \backslash G_{2n}$, and it follows from \thref{locality} and the fact that $G_{2n}$ is conformally removable that $f$ is actually conformal on $U_n$. Therefore, we have shown that $f$ is in fact conformal everywhere on $G$ with the possible exception of the point $p$. But a single point is conformally removable, and therefore $G$ is conformally removable.  

Let $E' = G \cap E$. It follows directly from the definition of conformal removability that a compact subset of a conformally removable set is also conformally removable. Therefore, $E' \subset G$ is conformally removable. It follows from \eqref{dimension-inequality-annuli} that
\be
\dim_H E' \geq d_{2n} + b_{2n} - 2
\ee
for all $n$. Since \be \lim_ { n \to \infty} d_{2n} + b_{2n} - 2 = d = \dim_H E, \ee  we have \be \dim_H E' \geq  \dim_H E. \ee Since $E' \subset E$, we conclude that $\dim_H E' = \dim_H E$ as desired.
\end{proof}

\newpage

\bibliographystyle{plain}
\bibliography{Removability.bib}

\end{document}